\RequirePackage{fix-cm}
\documentclass[smallextended]{svjour3}       
\smartqed  
\usepackage{graphicx}
\usepackage{bbold}
\usepackage{dsfont}
\usepackage{amscd}
\usepackage{amsmath}
\usepackage{amssymb}
\usepackage{mathrsfs}

\usepackage{bm}

\begin{document}

\title{On the control of a simplified $k-\varepsilon$ model of turbulence
\thanks{Partially supported by UESPI, Teresina - PI, Brazil, by CAPES Foundation, Ministry of Education of Brazil, Bras\'{i}lia DF 70040-020, Brazil and
 by grant MTM2016-76990-P, DGI-MICINN (Spain).}
}

\titlerunning{On the control of a simplified $k-\varepsilon$ model of turbulence}        

\author{Pit\'{a}goras~P.~Carvalho \\ Enrique~Fern\'{a}ndez-Cara \\ Juan~Limaco}

\authorrunning{P.P de Carvalho, J.B.L. Ferrel, E.Fernandez-Cara.} 

\institute{Pit\'{a}goras~P.~Carvalho \at
              Coor. Matem\'{a}tica - Universidade Estadual do Piau\'{i}, Teresina - PI, Brazil\\
              \email{pitagorascarvalho@gmail.com}           
           \and
           Enrique~Fern\'{a}ndez-Cara \at
           Dpto. E.D.A.N., Universidad de Sevilla, Aptdo. 1160, 41080 Sevilla, Spain \\
              Tel.: +34 954557992\\
              \email{cara@us.es}
           \and
            Juan~Limaco \at
          Inst. Matem\'{a}tica - Universidade Federal Fluminense, RJ, Brazil, \\
              Tel.:  +55 2126292111\\
              \email{jlimaco@vm.uff.br}
}

\date{Received: date / Accepted: date}

\maketitle

\begin{abstract}
This paper deals with the control of a kind of turbulent flows. We consider a simplified $k - \varepsilon$ model with distributed controls, locally supported in space. We proof that the system is partially locally null-controllable, in the sense that the velocity field can be driven exactly to zero if the initial state is small enough. The proof relies on an argument where we have concatenated several techniques: fixed-point formulation, linearization, energy and Carleman estimates, local inversion, etc. Ths result can be viewed as a nontrivial step towards the control of turbulent fluids.
\end{abstract}

\section{Introduction}
\label{Sec1}

   Let $\Omega \subset \mathds{R}^N$) be a bounded connected open set with a regular boundary $\partial\Omega$
 ($N=2$ or $N=3$ and let $\omega\subset \Omega$ be a (small) nonempty open set.
   
   Let $T>0$ be given and let us consider the cylindrical domain $Q=\Omega \times(0,T)$, with lateral boundary $\Sigma=\partial\Omega\times(0,T)$.
   In the sequel, $(\cdot\,,\cdot)$ and~$\| \cdot \|$ stand for the $L^{2}$ scalar product and norm in $\Omega$, respectively.
   We will denote by $C$ a generic positive constant; sometimes, we will indicate the data on which it depends. On the other hand, $n(x)$ will stand for the outwards unit normal vector to $\Omega$ at the point $x\in\partial\Omega$.

   We will investigate the local null controllability properties of a simplified model of turbulence of the~$k-\varepsilon$ kind.
   
   In order to present and justify the state system, let us start from the usual $k-\varepsilon$ model:
     \begin{equation}
\left\{
\begin{array}
[c]{ll}%
\vspace{.2cm}
\displaystyle v_{t} + (v\cdot\nabla)v - \nabla\cdot((\nu+c_{\nu}\frac{k^2}{\varepsilon})Dv) + \nabla(p-\frac{2}{3}k) = u1_{\omega} & \mbox{in}\;\;Q,\\
\vspace{.1cm}
\nabla\cdot{v} = 0 & \mbox{in}\;\;Q,\\
\vspace{.2cm}
\displaystyle k_{t} + v\cdot\nabla k - \nabla\cdot((\kappa+c_{0}\frac{k^2}{\varepsilon})\nabla k) = c_{\nu}\frac{k^2}{\varepsilon}|Dv|^2 - \varepsilon & \mbox{in}\;\;Q,\\
\vspace{.1cm}
 \displaystyle \varepsilon_{t} + v\cdot\nabla\varepsilon - \nabla\cdot((\kappa'+c'_{0}\frac{k^2}{\varepsilon})\nabla \varepsilon) = c_{\eta}k|Dv|^2 - a\frac{\varepsilon^{2}}{k} & \mbox{in}\;\;Q,\\
\vspace{.2cm}
v =  0 & \mbox{on}\;\;\Sigma,\\
\vspace{.2cm}
\displaystyle \frac{\partial k}{\partial n} = 0  \ \ \ \mbox{and} \ \ \ \frac{\partial \varepsilon}{\partial n} = 0 \ \ \ \ & \mbox{on}\;\;\Sigma,\\
\vspace{.1cm}
v(x,0) = v_{0}(x), \;\ \  \   k(x,0) = k_{0}(x) \;\  \ \  \mbox{and}  \;\  \ \  \varepsilon(x,0) = \varepsilon_{0}(x) \; \; \ & \mbox{in}\;\;\Omega.\\
\end{array}
\right.  \label{1.1}
   \end{equation}
\noindent

   Here, $v = v(x,t)$, $p=p(x,t)$, $k = k(x,t)$ and $\varepsilon = \varepsilon(x,t)$ are the ``averaged" velocity field, pressure, turbulent kinetic energy and rate of turbulent energy dissipation of a fluid whose particles are located in $\Omega$ during the time interval $(0,T)$; $v_{0}$, $k_{0}$ and $\varepsilon_{0}$ are the initial conditions at time $t = 0$; $1_{\omega}$ is the characteristic function of $\omega$; $\nu, \ c_{\nu}, \ \kappa, \ c_{0}, \ \kappa', \ c'_{0}, \ c_{\eta}$ and  $a$ are positive constants; see~\cite{17p}.

   The quantity $\nu_{T} := c_{\nu} k^2/\varepsilon$ is the so called turbulent viscosity, and $Dv$ stands for the symmetrized gradient of $v$, that is, $Dv := \nabla v + \nabla^{T} v$. Recall that $\nu_{T}$ appears in the averaged motion PDE as a consequence of the following Boussinesq hypothesis:
   $$
\overline{v' \otimes v'} = -\nu_{t}Dv - \frac{2}{3}k\,\mbox{Id.},
   $$
where, for any $z=z(x,t)$, $\bar{z}$ denotes the average of $z$ and $v'$ is the turbulent pertubation of the velocity field; for more details, see for instance \cite{T.Chacon}, \cite{G-O-P} and \cite{M.Lesieur}.

  The following vector spaces, usual in the context of incompressible fluids, will be used along the paper:
   $$
H = \{ w \in L^2(\Omega)^{N} : \nabla\cdot w = 0 \ \text{ in } \ \Omega, \;\  w\cdot n = 0 \ \text{ on } \ \partial\Omega \},
   $$
   and
   $$
V = \{ w \in H^1_{0}(\Omega)^{N} : \nabla\cdot w = 0 \ \text{ in } \ \Omega \}.
   $$

   We will denote by $A : D(A) \mapsto H$ the Stokes operator. By definition, one has
   $$
D(A) = H^2(\Omega)^{N} \cap V ; \;\  \ A w = P (-\Delta w) \ \ \ \forall w \in D(A),
   $$
where $P : L^2(\Omega)^{N} \mapsto H$  is the usual orthogonal projector.

   We will work with a simplified version of \eqref{1.1}.
   To this purpose, we introduce
   $$
\phi_{0} := \frac{k^{2}}{\varepsilon}
   $$
and we assume that $\phi_0$ only depends on $t$.
   As a consequence, after some straightforward computaions,we get the following PDEs and additional conditions for $ \displaystyle v, \ q:= p - \frac{2}{3} k$ and $k$:
     \begin{equation}
\left\{
\begin{array}[c]{ll}%
\vspace{.2cm}
\displaystyle v_{t} + (v\cdot\nabla)v - \nabla \cdot((\nu+c_{\nu}\phi_{0}) Dv) + \nabla q = u 1_{\omega} & \mbox{in}\;\;Q,\\
\nabla\cdot{v} = 0 & \mbox{in}\;\;Q,\\
\vspace{.1cm}
\displaystyle k_{t} + (v\cdot\nabla)k - \nabla \cdot (( \kappa + c_0\phi_{0}) \nabla k) + \frac{k^{2}}{\phi_{0}} =  c_{\nu}\phi_{0}|Dv|^2 \; \; \ & \mbox{in}\;\;Q,\\
\vspace{.2cm}
v = 0 & \mbox{on}\;\;\Sigma,\\
\vspace{.2cm}
\displaystyle \frac{\partial k}{\partial n} = 0 & \mbox{on}\;\;\Sigma,\\
\vspace{.2cm}
v(x,0) = v_{0}(x) \;\ \ \ \mbox{and} \ \ \ \   k(x,0) = k_{0}(x)  & \mbox{in}\;\;\Omega.\\
\end{array}
\right.  \label{1.2}
   \end{equation}

   Additionally, $\phi_0$ must satisfy the ODE problem:
   \begin{equation} \nonumber
\left\{
\begin{array}[c]{ll}
\vspace{.3cm}
\displaystyle \phi_{0,t} \!=\! \Biggl[\displaystyle \frac{2c_{\nu}\!-\!c_{\eta}}{|\Omega|} \Biggl(\int_\Omega \!\frac{|Dv|^2}{k} \,dx\Biggr) \!-\! \displaystyle \frac{2{c_{0}}}{|\Omega|}\Biggl(\int_\Omega \!\frac{|\nabla k|^2}{k^2} \,dx\Biggr)\Biggr] \phi_{0}^2 \!+\! \displaystyle \frac{a\!-\!2}{|\Omega|}\Biggl(\int_\Omega \! k\,dx\Biggr),  \\
\phi_{0}(0)  =  \phi_{00},
\end{array}
\right.
   \end{equation}
where $\phi_{00}$ is a positive initial datum.

   For simplicity, we will consider the previous system with $c_{\eta}=2c_{\nu}$ and $a\geq2$. Accordingly, one has
   \begin{equation} \label{1.3}
\left\{
\begin{array}
[c]{ll}%
\vspace{.3cm}
\displaystyle \phi_{0,t} = - \frac{2{c_{0}}}{|\Omega|} \Biggl(\int_\Omega \frac{|\nabla k|^2}{k^2} \ dx\Biggr) \phi_{0}^2 \ + \ \frac{a-2}{|\Omega|} \Biggl(\int_\Omega k \ dx \Biggr)  & \mbox{in}\;\; (0,T), \\
\phi_{0}(0)= \phi_{00} .
\end{array}
\right.
   \end{equation}

   In~\eqref{1.2}--\eqref{1.3}, $u$ is the control and $(v,q,k,\phi_0)$ is the state. When $N = 2$, for any $v_{0} \in H$, any nonnegative bounded $C^1$ function $\phi_{0}$ and any $ u \in L^2(\omega \times (0,T))^N$, \eqref{1.2}--\eqref{1.3}, possesses exactly one strong solution $(v, q)$, with
   \begin{equation}\label{1.4}
v\in L^2(0,T; D(A))\cap C^0([0,T];V), \ \ \ v_{t} \in L^2(0,T;H).
   \end{equation}
   When $N = 3$, this is true if $v_{0}$ and $u$ are sufficiently small in their respective spaces.
   These assertions can be deduced arguing (for instance) as in~\cite{Foias}.

\begin{definition} {\it Let $k_0$ and $\phi_{00}$ be given, with
   \begin{equation}\label{1.4a}
k_0 \in H^1(\Omega),  \ k_0\geq 0 \ \ \mbox{a.e.}, \ \phi_{00} \in \mathds{R}_+  .
   \end{equation}
   It will be said that \textsc{\eqref{1.2}--\eqref{1.3}} is partially locally null-controllable at time $T$ if there exists  $\epsilon > 0$ such that, for any $v_{0} \in V$  with
   $$
\| v_{0} \|_{H^1_0 } \leq \epsilon,
   $$
there exist controls $u \in L^{2}(\omega\times(0,T))^N $ and associated states $(v, q)$ satisfying \textsc{\eqref{1.4}} and
   \begin{equation}\label{1.1.1}
v(x,T) = 0 \ \text{ in } \ \Omega.
   \end{equation}
}
\end{definition} \label{D1}

   The main result in this paper is the following:

\begin{theorem}{For any $k_0$ and $\phi_0$ satisfying \textsc{\eqref{1.4a}} and any $T>0$, the nonlinear system \textsc{\eqref{1.2}--\eqref{1.3}} is partially locally null-controllable at  $T$.} \label{T1}
\end{theorem}

   For the proof, we will have to employ several different techniques, all them usual in this context nowadays. In particular, we will rewrite the partial null controllability problem as a fixed-point equation and we will apply an \textit{Inverse Mapping Theorem} in Hilbert spaces and then \textit{Schauder's Fixed-Point Theorem}. The arguments are inspired by the work of Fursikov and Imanuvilov; see~\cite{E.M3},  \cite{F-II} and \cite{F-I}.


   This paper is organized as follows. In Section~2, we prove some technical results, needed to establish the null contollability of \eqref{1.2}--\eqref{1.3}. In Section~3, we give the proof of Theorem \ref{T1}. Then, some additional comments and questions are presented in Sections~4.

\section{Preliminary results}\label{Sec2}

   In this section, we set $\mu := \nu + c_{\nu}{\Phi}_{0}(t)$, where $\Phi_0$ is a nonnegative function in~$C^1([0,1])$.
   We will consider the linear system
     \begin{equation}
\left\{
\begin{array}
[c]{ll}%
\vspace{.2cm}
\displaystyle v_{t} - (\nu+c_{\nu}\Phi_{0}) \Delta v + \nabla q = u 1_{\omega} + f \; \ \, &  \mbox{in}\;\;Q,\\
\vspace{.2cm}
\nabla\cdot{v} = 0 & \mbox{in}\;\;Q,\\
\vspace{.2cm}
v = 0 & \mbox{on}\;\;\Sigma,\\
v(x,0) = v_{0}(x)  \;\ \ \ \ & \mbox{in}\;\;\Omega\\
\end{array}
\right.  \label{2.1a}
   \end{equation}
and the adjoint
     \begin{equation}
\left\{
\begin{array}
[c]{ll}%
\vspace{.2cm}
-\varphi_{t} - (\nu+c_{\nu}\Phi_{0}) \Delta \varphi + \nabla\pi = F \; \ \, & \mbox{in}\;\;Q,\\
\vspace{.2cm}
\nabla\cdot{\varphi} = 0 & \mbox{in}\;\;Q,\\
\vspace{.2cm}
\varphi = 0 & \mbox{on}\;\;\Sigma,\\
\varphi(x,T) = {\varphi}_{T}(x)  \;\ \ \ \ & \mbox{in}\;\;\Omega .\\
\end{array}
\right.  \label{2.1b}
   \end{equation}

\subsection{Some Carleman estimates}
\noindent

   We will need some (well known) results from Fursikov and Imanuvilov~\cite{F-II}; see also~\cite{Enrique_Guerrero}.
   Also, it will be convenient to introduce a new non-empty open set $\omega_{0}$, with $\omega_{0} \Subset \omega$.
  The following technical lemma, due to Fursikov and Imanuvilov ~cite{F-II}, is fundamental:

\begin{lemma} {There exists a function $\eta^0 \in C^2(\overline{\Omega})$ satisfying:}\label{L1}
     \[
\left\{
\begin{array}
[c]{l}%
\vspace{.2cm}
\eta^0(x) > 0 \;\  \ \forall x \in \Omega, \;\  \ \eta^0(x) = 0 \;\  \ \forall x \in \partial\Omega \;\ \mbox{and} \\
|\nabla \eta^0(x)| > 0 \; \; \forall x \in \overline {\Omega}\setminus \omega_{0}.

\end{array}
\right.
\]
\end{lemma}

   Let $\tau = \tau(t)$ be a function satisfying
   \[
\tau\in C^{\infty}([0,T]), \ \ \tau > 0 \ \text{ in } \ (0,T), \ \ \tau(t) =
\left\{
\begin{array}{ll}
\displaystyle t  & \ \text{ if } \ t \leq \dfrac{T}{4}, \\
\noalign{\smallskip}
\displaystyle T - t & \ \text{ if } \ t \geq \dfrac{3T}{4}
\end{array}
\right.
   \]
and let us introduce the weights
   \begin{equation}
\left\{
\begin{array}{l}
\displaystyle \alpha(x,t):=\dfrac{e^{\lambda (\Vert\eta^{0}\Vert_{\infty} + m_{0} + 1)} -e^{\lambda(\eta^{0}(x) + m_{0})}}{\tau(t)^{8}}, \ \ \displaystyle \xi(x,t):=\frac{e^{\lambda(\eta^{0}(x) + m_{0})}}{\tau(t)^{8}}, \\
\noalign{\smallskip}
\displaystyle \rho(x,t):= e^{s \alpha(x, t)},
\ \ \overline{\rho}(t):= \exp \; (s \max_{x\in \ \overline{\Omega}}\alpha(x, t)),
\ \ \displaystyle \bar{\xi}(t):= \max_{ x \in \overline{\Omega}} \xi(x,t),
\end{array}
\right.\label{2.12}%
   \end{equation}
where $\lambda, \; s > 0$ are real numbers and the constant $m_{0} > 0$ is fixed, sufficiently large to have
   $$
|\alpha_{t}| \leq C\xi^{9/8}, \; \ |\alpha_{tt}| \leq C\xi^{5/4} \quad \forall\lambda > 0.
   $$
   It is then easy to prove that there exists $\lambda_{00} > 0$ such that, for any $\lambda \geq \lambda_{00},$ one has
   \begin{equation} \label{2.13}
\max_{x\in \ \overline{\Omega}}\alpha_{0}(x) \leq 2\min_{x\in \ \overline{\Omega}}\alpha_{0}(x).
   \end{equation}
   In the sequel, we will always take $\lambda \geq \lambda_{00}$.

   The following global Carleman estimate holds for the solutions to \eqref{2.1b}:

\begin{lemma} {Let us assume that $\Phi_0$ is a nonnegative function in $C^1([0,T])$, with $\|\Phi_0\|_{C^1([0,T])} \leq M$. There exist positive constants \ $\lambda_{0}, \; s_{0}$ and $C$ depending on $\Omega, \ \omega, \ T, \ \nu, \ c_{\nu},$ and $M$ such that, for any $s \geq s_{0}$ and $\lambda \geq \lambda_{0},$ any $F \in \; L^2(Q)$ and any $\varphi_{T} \in H,$ the associated solution to \emph{(\ref{2.1b})} satisfies}\label{L2}
\begin{align}\label{2.14}
\iint_Q{\overline{\rho}}^{\ -2} (s & \overline{\xi})^{-1} ( |\varphi_{t}|^2 + |\nabla \pi|^2 )\,dx\,dt + \iint_Q \rho^{-2} (s \xi)  |\nabla \times \varphi|^2 \,dx \,dt \nonumber \\
  + \iint_Q {\rho}^{-2}&((s \xi)^{-1}|\Delta \varphi|^2 + |\nabla \varphi|^2 + (s \xi)^2|\varphi|^2) \,dx\,dt \nonumber \\
 & \leq C \Biggl( \iint_Q \rho^{-2} |F|^2 \,dx\,dt + \iint_{\omega \times (0,T)} \rho^{-2}(s\xi)^3 |\varphi|^2 \,dx\,dt \Biggr).
\end{align}
\end{lemma}

   The proof is given in~\cite{I-P1}; see also~\cite{I-P}.

   By combining the previous lemma and appropriate energy estimates, we can deduce a second family of Carleman inequalities, with weights that do not vanish at $t=0$.
   More precisely, let us now introduce the function
   \[
\ell(t) := \left\{
\begin{array}{ll}
    \displaystyle \tau\Bigl(\dfrac{T}{2}\Bigr) = \dfrac{T}{2} \ \ & \mbox{ for } \ \displaystyle 0\leq t \leq \dfrac{T}{2} , \\
    \noalign{\smallskip}
    \displaystyle \tau(t) = T-t \ \ & \mbox{ for } \ \displaystyle\dfrac{3T}{4} \leq t \leq T
\end{array}
\right.
   \]
and the associated weights
   \begin{equation}
\left\{
\begin{array}{l}
\displaystyle\tilde{\alpha}(x,t):=\dfrac{e^{\lambda (\Vert\eta^{0}\Vert_{\infty} + m_{0} + 1)} -e^{\lambda(\eta^{0}(x) + m_{0})}}{\ell(t)^{8}}, \ \ \displaystyle \tilde{\xi}(x,t):=\dfrac{e^{\lambda(\eta^{0}(x) + m_{0})}}{\ell(t)^{8}}, \\
\noalign{\smallskip}
\displaystyle\tilde{\rho}(x,t):= e^{s \tilde{\alpha}(x, t)}, \ \rho_{\ast}(t):= \exp(s \max_{x \in \overline{\Omega}}\tilde{\alpha}(x, t)), \ \displaystyle \xi_{\ast}(t):= \max_{x \in \overline{\Omega}} \tilde{\xi}(x,t).
\end{array}
\right.  \label{2.15}%
   \end{equation}

   Then the functions $\tilde{\xi}, \; \tilde{\rho}, \; \rho_{\ast}$ and $ \xi_{\ast} $ are strictly positive and bounded from below in any set of the form $\overline{\Omega}\times[0, T-\delta]$ with $\delta > 0$.

\begin{lemma}{Let the assumptions in Lemma \textsc{\ref{L2}} be satisfied. There exist positive constants $\lambda_{0}, \ s_{0}$ and $C$ depending on $\Omega$, $\omega$, $T$, $\nu$, $c_{\nu}$ and $M$ such that, for any $s \geq s_{0}$ and $\lambda \geq \lambda_{0},$ any $F$ $\in \; L^2(Q)$ and any $\varphi_{T} \in \; H,$ the associated solution to \emph{\eqref{2.1b}} satisfies}\label{L3}
\begin{align}\label{2.16}
\iint_Q \rho_{\ast}^{-2} (s&\xi_{\ast})^{-1} ( |\varphi_{t}|^2 + |\nabla \pi|^2 ) \,dx\,dt + \iint_Q \tilde{\rho}^{-2} (s \tilde{\xi})  |\nabla \times \varphi|^2 \,dx\,dt \nonumber \\
 + \iint_Q \tilde{\rho}^{-2}& \biggl((s \tilde{\xi})^{-1}|\Delta \varphi|^2 + |\nabla \varphi|^2 + (s \tilde{\xi})^2|\varphi|^2 \biggr)\,dx\,dt  \\
 & \leq C \Biggl( \iint_Q \tilde{\rho}^{ \ -2} |F|^2 \,dx\,dt + \iint_{\omega \times (0,T)} \tilde{\rho}^{ \ -2}(s \tilde{\xi})^3 |\varphi|^2 \,dx\,dt \Biggr). \nonumber
\end{align}
\end{lemma}

\begin{proof} It is a consequence of \eqref{2.14} and the usual energy estimates for $\varphi.$ Since the main ideas are known and can be found in many papers, we will only give a sketch.

Let us take $\lambda \geq \lambda_{0}$ and $ s \geq s_{0} $ and let us set
\begin{align}
 \Gamma (\varphi, \pi; a, b)&:=  \iint_{\Omega\times(a,b)} \rho_{\ast}^{-2} (s \xi_{\ast})^{-1} ( |\varphi_{t}|^2 + |\nabla \pi|^2 ) \,dx\,dt \nonumber\\
+ \ &\iint_{\Omega\times(a,b)} \tilde{\rho}^{\ -2} \bigl((s \tilde{\xi})^{-1}|\Delta \varphi|^2 + |\nabla \varphi|^2 + (s \tilde{\xi})^2|\varphi|^2 \bigr)\,dx\,dt \nonumber \\
+ & \iint_{\Omega\times(a,b)} \tilde{\rho}^{\ -2} (s \tilde{\xi})  |\nabla \times \varphi|^2 \,dx\,dt
\end{align}
and
\begin{align}
S(\varphi; F) := & \iint_Q \tilde{\rho}^{\ -2} |F|^2 \,dx\,dt + \iint_{\omega \times (0,T)} \tilde{\rho}^{\ -2}(s\tilde{\xi})^3 |\varphi|^2 \,dx\,dt. \nonumber
\end{align}


Then $\Gamma (\varphi, \pi; 0, T) = \Gamma (\varphi, \pi; 0, T/2) + \Gamma (\varphi, \pi; T/2, T)$ and, from Lemma \ref{L2}, we clearly have
\begin{equation}\label{2.13a}
  \Gamma (\varphi,\pi; T/2, T) \leq C S(\varphi; F).
\end{equation}

We must prove that
$$ \Gamma(\varphi,\pi; 0, T/2) \leq C (S(\varphi; F)$$
and we know that
$$ \Gamma (\varphi,\pi; 0, T/2) \leq C\int_{0}^{T/2} ( \| \varphi_{t} \|^2 + \| \Delta\varphi \|^2 + \| \nabla \varphi \|^2 + \| \varphi \|^2 )\,dt. $$
Let us check that this integral can be bounded as in \eqref{2.13a}.

From (\ref{2.1b}), it is readily seen that
$$  -\frac{1}{2}\frac{d}{dt} \|\varphi\|^2  + \mu \|\nabla\varphi\|^2  \leq \frac{\mu}{2} \|\nabla \varphi\|^2 + C\|F\|^2  \; \ \ \mbox{in} \; \ \ (0,T), $$
whence
\begin{align}\label{2.15}
&\|\varphi(\cdot, t_{1}) \|^2  + \int_{t_{1}}^{t_{2}} \|\nabla\varphi(\cdot, s) \|^2 \ ds \leq  \|\varphi(\cdot, t_{2}) \|^2 + C  \int_{t_{1}}^{t_{2}} (\| F (\cdot, s) \|^2  \nonumber
\end{align}
for all $ 0 \leq t_{1} \leq t_{2} \leq T.$ In particular, taking $t_{1} \in [0, T/2]$ and $ t_{2} \in [T/2, 3T/4]$ and integrating with respect to $t_{1}$ and then with respect to $t_{2}$, we see that
\begin{align}
\vspace{.2cm}
  \displaystyle \int_{0}^{T/2} \|\varphi(\cdot, t)\|^2 \,dt & \leq C \biggl(\int_{T/2}^{3T/4}\|\varphi(\cdot, s)\|^2 \,ds +  \int_{0}^{3T/4} \|F(\cdot, s) \|^2 \, ds \biggr) \nonumber \\
 & \leq C S(\varphi; F) \nonumber.
\end{align}

Also,
\begin{center}
$ \displaystyle \int_{0}^{T/2} \|\nabla \varphi(\cdot, t)\|^2 \ dt \leq  \|\varphi(\cdot, t_{2}) \|^2 \  + C \int_{0}^{3T/4} \|F(\cdot, s) \|^2  \; ds , $ \\
\end{center}
for all $t_{2} \in [T/2 , 3T/4]$ and, integrating with respect to $t_{2}$ in this interval, we deduce an estimate of the integral of $\|\nabla \varphi\|^2$ in $(0, \, T/2)$:
\begin{align}
\vspace{.2cm}
 \displaystyle \int_{0}^{T/2} \|\nabla \varphi(\cdot, t)\|^2 \,dt & \leq  C \biggl( \int_{T/2}^{3T/4} \|\varphi(\cdot, s) \|^2 \,ds +  \int_{0}^{3T/4} \|F(\cdot, s) \|^2  \; ds \biggr) \nonumber \\
& \leq C S(\varphi; F) .  \nonumber
\end{align}
A similar argument holds for the integral of $\|\Delta\varphi \|^2$.  Indeed, one has
$$  -\frac{1}{2}\frac{d}{dt} \|\nabla\varphi\|^2  + \mu \|\Delta\varphi\|^2 \ = \ (F, - \Delta \varphi) \ \leq \ \frac{\mu}{2} \|\Delta\varphi\|^2 + C\|F\|^2 \ \ \; \ \mbox{in} \ \ \; \ (0,T). $$
Arguing as before, we see that
\begin{center}
$ \displaystyle \int_{0}^{T/2} \|\Delta\varphi(\cdot, t)\|^2 \,dt \leq   \|\nabla\varphi(\cdot, t_{2}) \|^2 + C \int_{0}^{3T/4} \|F(\cdot, s) \|^2 \; ds$
\end{center}
for all $t_{2} \; \in \; [T/2, 3T/4]$ and, after integration with respect to $t_{2}$, we see that
\begin{align}
\displaystyle \int_{0}^{T/2} \|\Delta\varphi(\cdot, t)\|^2 \,dt & \leq  C \biggl( \int_{T/2}^{3T/4} \|\nabla\varphi(\cdot, s) \|^2 \; ds + \int_{0}^{3T/4} \|F(\cdot, s) \|^2 \; ds \biggr) \nonumber \\
& \leq C S(\varphi; F) \nonumber .
\end{align}
Finally, using that $\|\varphi_{t}\|^2 = (\mu\Delta\varphi + F, - \varphi_{t} )$, we deduce that
   $$
\int_{0}^{T/2} \|\varphi_{t}\|^2  \,dt \leq C S (\varphi; F).
   $$

   As a consequence, $ \Gamma (\varphi,\pi; 0, T/2) \leq C S(\varphi; F)$ and the proof is done.
\end{proof}

   From now on, we will fix $\lambda$ and $s$ as in Lemma \ref{L3} and we will consider the corresponding functions $\alpha, \ \tilde{\alpha}, \ \rho, \ \xi$, etc. given by (\ref{2.12}) and (\ref{2.14}). Also, we will introduce the function
\begin{equation}\label{2.21}
\tilde{\eta} := \tilde{\rho} \; \tilde{\xi}^{-1}.
\end{equation}

An immediate consequence of Lemma \ref{L3} and this definition is the following:

\begin{corollary}{ There exist positive constants $\lambda, \; s$ and $C$ only depending on $\Omega,$ $\omega,$ $ T $, $\nu$, $c_{\nu}$ and $M$  such that, for any $F\in \; L^2(Q)$ and any $\varphi_{T} \in \; H,$ the corresponding solution to \emph{(\ref{2.1b})} satisfies}
\begin{align}\label{2.22}
&\iint_{Q} \tilde{\eta}^{-2} |\varphi|^2  \,dx \,dt \ \leq \ C \Biggl( \iint_Q \tilde{\rho}^{-2} |F|^2 \,dx \,dt + \iint_{\omega \times (0,T)} \tilde{\eta}^{-2} |\varphi|^2 \,dx\,dt \Biggr). \nonumber
\end{align}
\end{corollary}
\noindent

\subsection{ The null controllability of the linear system (\ref{2.1a})}
\noindent

 We will have to impose some specific conditions to $f$ and $v_{0}$ in order to drive the solution to (\ref{2.1a}) exactly to zero. The following result holds:

\begin{proposition}{Let us assume that $v_{0} \in  H \ \mbox{and} \ \  \tilde{\eta}f  \in L^2(Q)^N $.
Then, we can find control-states $(v, q, u) $ for $\emph{(\ref{2.1a})}$ satisfying}\label{P1}
\begin{equation}\label{2.23}
  \ v \in L^2(0,T;V)\cap C^0([0,T]; H), \; \ u \in L^2(\omega \times (0, T))^N
\end{equation}
and
\begin{align}\label{2.24}
&\iint_{Q} \tilde{\rho}^{2} |v|^2  \,dx\,dt \ + \ \iint_{\omega \times (0,T)} \tilde{\eta}^{ 2} |u|^2 \,dx\,dt \ < \ + \infty.
\end{align}
In particular, one has \textsc{(\ref{1.1.1})}.
\end{proposition}

The proof relies on (\ref{2.15}) and can be easily obtained arguing as in \cite{F-II}.
The key idea is to consider the extremal problem

\begin{equation}
\left\{
\begin{array}[c]{ll}

\displaystyle
    \mbox{Minimize} \iint_{Q} \tilde{\rho}^{2} |v|^2  \,dx \,dt + \iint_{\omega \times (0,T)} \tilde{\eta}^{2} |u|^2  \,dx\,dt \\
    \\
\displaystyle
   \mbox{Subject  to} \; \  u \in L^2(\omega \times (0,T))^N, \; \; (v, q, u) \; \ \mbox{satisfies} \; \ (\ref{2.1a}).\\
\end{array}
\right.  \label{2.25}%
\end{equation}

There exists exactly one solution to \eqref{2.25} that, thanks to \eqref{2.16}, is the desired control-state and satisfies \eqref{2.23} and \eqref{2.24}.

It is also true that, in a certain sense, the solution to (\ref{2.25}) depends continuously on the data $(f, v_{0})$. In particular, if the $(f_n, v_{0n})$ satisfy
$$ \|v_{0n} - v_{0} \| \rightarrow 0  \; \ \; \ \mbox{and} \; \ \;  \iint_{Q}\tilde{\eta}^2 |f_{n} - f|^2 \,dxdt \rightarrow 0, $$
then the $v_{n}$ and $u_{n}$  furnished by Proposition \ref{P1} satisfy
\begin{equation} \label{2.26}
    \iint_{Q} \tilde{\rho}^{\ 2} |v_{n} - v|^2 \,dx\,dt \rightarrow 0  \; \ \; \ \mbox{and} \; \ \; \iint_{\omega \times (0,T)} \tilde{\eta}^{\ 2} |u_{n} - u|^2 \,dx\,dt \rightarrow 0,
\end{equation}
 where $u$ and $v$ correspond to $(f, v_{0})$ . For brevity we omit the details, that can be found in \cite{E.M3}.

\subsection{ Some additional estimates }
\noindent

The state found in Proposition \ref{P1} satisfies some additional properties that will be needed below, in Section~3. Thus, it will be shown in this section that not only $v$ but also $\nabla v,  \; \Delta v, \; \mbox{and} \ \  v_{t}$
belong to some specific weighted $L^2$~spaces.

Let us introduce the spatially homogeneous weights
\begin{equation}\label{2.27}
\hat{\rho}(t) := \exp (s \min_{x\in \ \overline{\Omega}}\tilde{\alpha}(x,t)), \ \ \zeta(t) := \hat{\rho}(t)\ell(t)^{12}
\ \mbox{ and } \ \gamma(t) := \hat{\rho}(t)\ell(t)^{33/2}.
\end{equation}
   From (\ref{2.15}), (\ref{2.21}) and (\ref{2.27}), it is easy to see that
   \begin{equation}\label{2.28}
\zeta \leq C \tilde{\eta} \ \mbox{ and } \ |\zeta \zeta_{t} |\leq C \tilde{\rho}^{\ 2}.
   \end{equation}
   
   The following results hold:

\begin{proposition} {Let the hypotheses in Proposition \textsc{\ref{P1}} be satisfied and let $(v, q, u)$ be the solution to \textsc{(\ref{2.25})}.}\label{P2}
Then
\begin{equation}
\left\{
\begin{array}{l}
\displaystyle \max_{[0,T]}  \int_{\Omega} \!\zeta^2|v|^2 \,dx \!+\! \iint_{Q} \!\zeta^2|\nabla v|^2 \,dx\,dt \!\leq\! C \biggl( \| v_{0} \|^2  \!+\! \displaystyle \iint_{Q} \!\tilde{\rho}^{\ 2}|v|^2 \,dx \;dt \\
\noalign{\smallskip}
\displaystyle
\phantom{\displaystyle \max_{[0,T]}  \int_{\Omega} \!\zeta^2|v|^2 \,dx} \!+\! \displaystyle \iint_{Q}\tilde{\eta}^2|f|^2 \,dx \;dt  + \iint_{\omega \times (0,T)}\tilde{\eta}^2|u|^2 \,dx \;dt\biggr)
\end{array}
\right.  \label{2.29}%
\end{equation}
\end{proposition}

\begin{proof} To get this result, we multiply the PDE in (\ref{2.1a}) by $\zeta^2v$ and we integrate in $\Omega$. We obtain:
$$ \int_{\Omega}\zeta^2 (v_{t} - \mu\Delta v + \nabla q)\cdot v \,dx = \int_{\Omega}\zeta^2 ( u1_{\omega} + f) \cdot v \,dx.  $$
Remember that $\mu=\nu+c_{\nu}\Phi_{0}(t)$.

In view of the inequalities in (\ref{2.28}), the following estimates hold:
\begin{align}
 \int_{\Omega}\zeta^2 u 1_{\omega}\cdot v \,dx \ & \leq \ C \biggl( \int_{\omega}\tilde{\eta}^2 |u|^2 \,dx \biggr)^{1/2}\biggl( \int_{\Omega} \zeta^4 \tilde{\eta}^{-2} |v|^2 \,dx \biggr)^{1/2} \nonumber\\ \vspace{.2cm}
 & \leq \, \ \frac{1}{2} \int_{\omega}\tilde{\eta}^2 |u|^2 \,dx + C \int_{\Omega}\tilde{\rho}^{\ 2} |v|^2 \,dx \, , \nonumber\\ \nonumber \\
 \displaystyle \int_{\Omega}\zeta^2 f \cdot v \,dx \ & \leq \ \frac{1}{2} \int_{\Omega}\tilde{\eta}^{ \ 2} |f|^2 \,dx + C \int_{\Omega}\tilde{\rho}^{\ 2} |v|^2 \,dx \, ,\nonumber\\ \nonumber \\
 \displaystyle \int_{\Omega}\zeta^2 v_{t}\cdot v \,dx & \ = \ \frac{1}{2} \frac{d}{dt} \int_{\Omega}\zeta^2 |v|^{2} \,dx - \int_{\Omega}\zeta \; \zeta_{t} |v|^2 \,dx \nonumber\\
  & \, \geq \ \frac{1}{2} \frac{d}{dt}\int_{\Omega}\zeta^2 |v|^2 \,dx - C \int_{\Omega}\tilde{\rho}^{ \ 2} |v|^2 \,dx \nonumber.
\end{align}
\\
On the other hand,
   $$
\int_{\Omega} \zeta^2\nabla q \cdot v \,dx\,dt\; = 0  \; \ \; \ \mbox{and} \; \ \; \ \int_{\Omega} \zeta^2 (-\Delta v) \cdot v \;dx\; = \int_{\Omega} \zeta^2|\nabla v|^2 \;dx.
   $$
   Therefore,
\begin{align}
\displaystyle \frac{1}{2}\frac{d}{dt}\int_{\Omega}\zeta^2|v|^2 \,dx + \mu \int_{\Omega}\zeta^2|\nabla v|^2 \,dx \leq & C \biggl( \int_{\omega}\tilde{\eta}^2|u|^2 \,dx \nonumber \\
& + \int_{\Omega}\tilde{\rho}^2|v|^2 \,dx + \int_{\Omega}\tilde{\eta}^2 |f|^2 \,dx\biggr). \nonumber
\end{align}
Now, integrating in time, the estimate in (\ref{2.29}) is easily found.
\end{proof}

\begin{proposition}\label{P3} Let the hypotheses in~Proposition~\ref{P1} be satisfied and let $(v, q, u)$ be the solution to~\eqref{2.25}. Let us assume that $v_{0} \in V$. Then one has
\begin{equation}
\left\{
\begin{array}{l}
\displaystyle \max_{[0,T]} \int_{\Omega} \gamma^2|\nabla v|^2 \;  dx + \iint_{Q}\gamma^2(|v_{t}|^2+  |\Delta v|^2)\,dx\,dt \leq C \biggl( \| v_{0} \|_{H_{0}^1}^2 \\
\noalign{\smallskip}
\displaystyle \ \ + \displaystyle \iint_{Q}\tilde{\rho}^2|v|^2 \,dx\,dt  + \displaystyle \iint_{Q}\tilde{\eta}^2|f|^2 \,dx \,dt + \iint_{\omega \times (0,T)}\tilde{\eta}^2|u|^2 \,dx \;dt\biggr).
\end{array}
\right.  \label{2.30a}%
\end{equation}
\end{proposition}

\begin{proof}
   Recall that, under the assumption $v_{0}\in V,$ if $\Phi_{0}$ is a nonnegative $C^1$~function in~$[0,T]$ and~$f\in L^2(Q)^N$, the solution $(v,q)$ to (\ref{2.1a}) satisfies
$$ v \in L^2(0,T; D(A))\cap C^0([0,T]; V), \; \ \; \ \ v_{t} \in L^2(0,T; H), $$
where $A : D(A) \mapsto H$ is the Stokes operator.

Let us multiply the PDE in (\ref{2.1a}) by $ \gamma^2 v_{t} $ and let us integrate in $\Omega$. The following holds:
$$ \int_{\Omega} \gamma^2 ( v_{t} - \mu\Delta v + \nabla q) \cdot v_{t} \,dx = \int_{\Omega} \gamma^2 ( u 1_{\omega} + f) \cdot v_{t} \,dx .$$
From the definition of $\gamma$, $\zeta$ and $\tilde{\eta}$ we see that
$$ \gamma \leq C\zeta \leq C \tilde{\eta} \;;  \; \ \; \ \; \ |\gamma \gamma_{t}|\leq C \zeta^2.  $$
Consequently, for any small $\epsilon > 0,$ we find that
$$  \int_{\Omega}\gamma^2 u1_{\omega}\cdot v_{t} \,dx \leq \epsilon\int_{\Omega}\gamma^2 |v_{t}|^2 \,dx + C_{\epsilon}\int_{\omega}\tilde{\eta}^2|u|^2 \,dx, $$
$$  \int_{\Omega}\gamma^2 f \cdot v_{t} \,dx \leq \epsilon \int_{\Omega}\gamma^2 |v_{t}|^2 \,dx + C_{\epsilon}\int_{\Omega}\tilde{\eta}^2|f|^2 \,dx $$
and, also,
\begin{align}
  - \int_{\Omega}\gamma^2 \Delta v\cdot v_{t} \,dx & = \ \frac{1}{2}\frac{d}{dt}\int_{\Omega}\gamma^2|\nabla v|^2 \,dx - \int_{\Omega} \gamma\gamma_{t}|\nabla v|^2 \;dx \nonumber \\
   & \geq \,  \frac{1}{2}\frac{d}{dt}\int_{\Omega}\gamma^2|\nabla v|^2 \,dx - C \int_{\Omega}\zeta^2|\nabla v|^2 \,dx. \nonumber
\end{align}

 On the other hand, the integral involving the pressure vanishes. Therefore, the following is found integrating in time:
\begin{align}\label{2.31}
\iint_{Q} & \gamma^2 |v_{t}|^2 \,dx\,dt + \max_{[0,T]}\int_{\Omega}\gamma^2|\nabla v |^2 \,dx \leq C \biggl(\|v_{0}\|^2_{H_{0}^1(\Omega)} \nonumber \\
& +\! \iint_{Q}\zeta^2 |\nabla v|^2 \,dx \,dt \!+\! \iint_{Q}\tilde{\eta}^2 |f|^2 \,dx\,dt \!+\! \iint_{\omega \times (0,T)} \tilde{\eta}^2 |u|^2 \,dx\,dt\biggr).
 \end{align}

 This furnishes the estimates of $\gamma^2|v_{t}|^2$ and $\gamma^2|\nabla v|^2$ in \eqref{2.30a}.

 In order to estimate the weighted integral of $|\Delta v|^2$, we multiply the PDE in (\ref{2.1a}) by $\gamma^2 Av$ (recall that $A$ is the Stokes operator). After integration in $\Omega$, we have:
 $$ \int_{\Omega} \gamma^2 (v_{t} - \mu\Delta v + \nabla q)\cdot Av \,dx = \int_{\Omega} \gamma^2 ( u1_{\omega} + f) \cdot Av \,dx. $$
 Observe that
 \begin{align}\label{2.321} \nonumber
 &  \int_{\Omega}\gamma^2 u 1_{\omega}\cdot Av \,dx \leq \epsilon \int_{\Omega}\gamma^2|\Delta v|^2 \,dx + C_{\epsilon}\int_{\omega}\tilde{\eta}^2|u|^2 \,dx, \\ \nonumber
 &  \int_{\Omega}\gamma^2 f \cdot Av \,dx \leq \epsilon \int_{\Omega}\gamma^2|\Delta v|^2 \,dx + C_{\epsilon}\int_{\Omega}\tilde{\eta}^2|f|^2 \,dx, \\  \nonumber
 &  \int_{\Omega}\gamma^2 v_{t}\cdot Av \,dx = \int_{\Omega}\gamma^2 v_{t}\cdot (-\Delta v) \,dx \leq \epsilon \int_{\Omega}\gamma^2|\Delta v|^2 \,dx + C_{\epsilon}\int_{\Omega}\gamma^2|v_{t}|^2 \,dx  \nonumber
\end{align}
for any small $\epsilon > 0$,
$$\int_{\Omega}\gamma^2 \nabla q \cdot Av \,dx = 0 \; \ \; \  \mbox{and} \; \ \; \ \int_{\Omega}\gamma^2 (- \Delta v) \cdot Av \,dx = \int_{\Omega}\gamma^2 |\Delta v|^2 \,dx. $$
Integrating in time, we now see that
\begin{align}
\displaystyle \iint_{Q} \gamma^2|\Delta v|^2 \,dx\;dt & \leq C \biggl( \iint_{\omega \times(0,T)}\tilde{\eta}^{ \ 2}|u|^2 \,dx \;dt \nonumber \\
& + \iint_{Q}\tilde{\eta}^2 |f|^2 \,dx \;dt + \iint_{Q}\gamma^2|v_{t}|^2 \,dx \;dt \biggr).\nonumber
\end{align}
From (\ref{2.31}) and this last inequality, we deduce (\ref{2.30a}).
\end{proof}

\section{ Proof of Theorem \ref{T1}}\label{Sec3.}
\noindent

In this section, we will prove the partial local null controllability of (\ref{1.2}).

Let us denote by $L^2(\tilde{\eta}^2; Q)$ the Hilbert space formed by the measurable functions $f = f(x,t)$ such that~$\tilde{\eta} f\in L^2(Q)$, that is,
   $$
\|f\|^2_{L^2(\tilde{\eta}^2; Q)} := \iint_Q \tilde{\eta}^2|f|^2 \;dx\;dt < +\infty.
   $$

Let $\Phi_0 \in C^1([0,T])$ be given, with $\Phi_0 \geq 0$ and $\|\Phi_0\|_{C^1([0,T])} \leq M$. Let us set
   \begin{equation}\label{3.32}
L(\Phi_0)v := v_{t} - (\nu + c_{\nu}\Phi_0) \Delta v
   \end{equation}
and let us introduce the Hilbert spaces
   \begin{align}\label{3.33}
W(\Phi_0) :=  \Bigl\{  (v,& \, q, \, u)  : \; \tilde{\rho}v \in L^2(Q)^N,  \; \; \gamma v \in L^2(0,T; D(A)), \nonumber \\
   \tilde{\eta}u & \in L^2(\omega \times (0,T))^N, \ q \in L^2(0,T; H^1(\Omega)), \\
   \noalign{\smallskip}
   \ \int_{\Omega} q &\,dx = 0  \ \; a.e. \; \ \tilde{\eta}\bigl(L(\Phi_0) + \nabla q - u 1_{\omega}\bigr) \in L^2(Q)^N \Bigr\} \nonumber
   \end{align}
and
   \begin{equation}
Z :=  L^2(\tilde{\eta}^2; Q)^N \times V \;.
   \end{equation}

We will use the following Hilbertian norms in $W(\Phi_0)$ and $Z$:
\begin{equation}
\begin{array}[c]{ll}
\displaystyle
\| (v, q, u) \|^2_{W(\Phi_0)} & :=  \|\tilde{\rho} \;v\|^2_{L^2(Q)}+ \|\gamma v\|^2_{L^2(0,T; D(A))} + \|\tilde{\eta} \;u\|^2_{L^2(\omega \times (0,T))} \\
\noalign{\smallskip}
& + \, \|q\|^2_{L^2(0,T; H^1(\Omega))}
+ \|\tilde{\eta}\;(L(\Phi_0) + \nabla q - u 1_{\omega})\|^2_{L^2(Q)^N}
\end{array}
\end{equation}

and

\begin{equation}
\begin{array}[c]{ll}
\displaystyle
   \|(f, v_0)\|^2_{Z} := & \|f\|^2_{L^2(\tilde{\eta}^2; Q)} + \|v_0\|^2_{H^1_0} \ .
\end{array} \label{2.30}%
\end{equation}

Note that, if \textsc{$(v, q, u)$} $\in W(\Phi_0),$ then $v_{t} \in L^2(Q)^N,$ whence $v : [0,T] \mapsto V$  is (absolutely) continuous and, in particular, we have $v(\cdot \, , 0) \in V$  and
$$ \| v(\cdot\, , 0) \|_{V}  \leq C(M) \|(v, q, u)\|_{W(\Phi_0)} \; \; \ \forall \; (v, q, u) \in W(\Phi_0) .  $$
Furthermore, in view of Propositions~\ref{P2} and~\ref{P3}, one also has $\zeta v \in L^2(0,T; V) \cap L^{\infty}(0,T; H)$ and $\gamma v \in L^2(0,T; D(A)) \cap L^{\infty}(0,T; V),$ with norms bounded again by $C(M)\|(v, q, u)\|_{W(\Phi_0)} \ .$

In the first part of this section, we will assume that the turbulent viscosity is given and we will try to drive the (averaged) velocity field exactly to zero. Then, we will use a fixed-point argument to prove that the whole system \eqref{1.2}--\eqref{1.3} is partially locally null-controllable.

Thus, let us consider the mapping $\mathcal{A} : W(\Phi_0) \mapsto Z,$ given as follows:
\begin{equation}\label{3.34}
\mathcal{A}(v, q, u) := \bigl(v_{t} + (v\cdot\nabla)v - (\nu+c_{\nu}\Phi_{0}) \Delta v + \nabla q - u 1_{\omega} \ , \ v(\cdot\, , 0)  \bigr).
\end{equation}

We will check that there exists $\epsilon > 0$ such that, if $(f, v_{0}) \in Z$ and $\|(f, v_{0})\|_{Z} \leq \epsilon $, the equation
 \begin{equation}\label{3.39}
 \mathcal{A}(v, q, u) = (f, v_{0}), \; \ \; \ \ (v, q, u)\in W(\Phi_0),
 \end{equation}
 possesses at least one solution. In particular, this will show that the nonlinear system
       \begin{equation}
\left\{
\begin{array}
[c]{ll}%
\vspace{.2cm}
v_{t} + (v \cdot \nabla)v - (\nu+c_{\nu}\Phi_{0}) \Delta v + \nabla q = u 1_{\omega} + f \; \  & \mbox{in}\;\;Q,\\
\vspace{.2cm}
\nabla\cdot{v} = 0 & \mbox{in}\;\;Q,\\
\vspace{.2cm}
v = 0 & \mbox{on}\;\;\Sigma,\\
\vspace{.2cm}
v(x,0) = v_{0}(x)  \;\ \ \ \ & \mbox{in}\;\;\Omega\\
\end{array}
\right.  \label{3.31a}
   \end{equation}
is \emph{locally null controllable }and, furthermore, the state-controls $(v, q, u)$ can be chosen in $W(\Phi_0)$.

We will apply the following version of the \emph{Inverse} \emph{Mapping} \emph{Theorem} in infinite dimensional spaces that can be found for instance in \cite{A.1}.
In the following statement, $B_r(0)$ and $B_{\epsilon}(\zeta_0)$ are open balls respectively of radius~$r$ and~$\epsilon$.

\begin{theorem}\label{T3}
   Let  Y and Z be Banach spaces and let $H : B_{r}(0) \subset Y \mapsto Z$ be a $C^1$ mapping. Let us assume that the derivative $H'(0) : Y \mapsto Z$ is an epimorphism and let us set $\zeta_{0} = H(0)$.
   Then there exist $\epsilon > 0$, a mapping $S : B_{\epsilon}(\zeta_{0}) \subset Z \mapsto Y$ and a constant $K > 0$ satisfying:
   \begin{equation}
\left\{
\begin{array}
[c]{ll}%
\vspace{.2cm}
S(z) \in B_{r}(0) \; \ \mbox{and}  \; \ H(S(z)) = z \; \ \; \ \forall z \in B_{\epsilon}(\zeta_{0}), \\
\|S(z)\|_{Y} \leq K \|z - H(0)\|_{Z} \; \ \; \ \forall z \in B_{\epsilon}(\zeta_{0}).  \\
\end{array}
\right.  \label{3.31b}
   \end{equation}
\end{theorem}

   Notice that, in this result, usually known as {\it Liusternik's Inverse Function Theorem,} $S$ is the inverse to the right of~$H$.
   In order to show that Theorem \ref{T3} can be applied in this setting, we will use several lemmas.

\begin{lemma} {Let $\mathcal{A} : W(\Phi_0) \mapsto Z$ be the mapping defined by \emph{(\ref{3.34})}. Then, $\mathcal{A}$ is well defined and continuous.}\label{L4}
\end{lemma}

\begin{proof} First, note that, in view of (\ref{2.27}) and (\ref{2.13}) we have
   \begin{equation}\label{3.40}
\tilde{\eta}^2 \leq C\zeta\gamma^3 \leq C\gamma^6 .
   \end{equation}
   Notice that
   \begin{align}
\|\mathcal{A}(v, q, u)\|_{Z}^2 & = \iint_{Q}\tilde{\eta}^2|v_{t} + (v \cdot \nabla )v 
- (\nu + c_{\nu}\Phi_{0})\Delta v + \nabla q - u1_{\omega} |^2 \;dx \;dt \nonumber \\
& + \|v(\cdot,0)\|^2_{H_0^1} \, . \nonumber
   \end{align}
The following holds:
\begin{align}\label{3.41}
\iint_{Q}&\tilde{\eta}^2|v_{t} + (v \cdot \nabla )v - (\nu + c_{\nu}\Phi_{0})\Delta v + \nabla q - u1_{\omega} |^2 \;dx\;dt  \nonumber \\
& \leq 2 \iint_{Q}\tilde{\eta}^2|v_{t} - (\nu + c_{\nu} \Phi_{0})\Delta v + \nabla q - u1_{\omega} |^2 \;dx\;dt  \nonumber \\
& + 2 \iint_{Q}\tilde{\eta}^2|(v \cdot \nabla)v|^2 \,dx\;dt \nonumber\\
& := M_1 + M_2 \, .
\end{align}
From the definitions of the space $W(\Phi_0)$ and the norm $\|(v, q, u)\|_{W(\Phi_0)}$, it follows that
\begin{align}  \nonumber
M_1 \leq C\|(v, q, u)\|^2_{W(\Phi_0)} \ . \nonumber
\end{align}
On the other hand, taking into account that for any $w \in D(A)$ one has
$$\|\nabla w\|_{L^3} \leq C \|\nabla w\|^{1/2}\|\Delta w\|^{1/2}$$
 and  
$$\|(w \cdot \nabla)w \|^2 \leq C \|w\|^2_{L^6}\|\nabla w\|^2_{L^3} \leq C \|\nabla w\|^3\|\Delta w\|,$$
we see that
\begin{align}\label{3.33a}
M_2 \leq & \; C \|\zeta v\|^{1/2}_{L^2(0,T; V)} \|\gamma v\|_{L^{\infty}(0,T; V)} \|\gamma v\|^{1/2}_{L^2(0,T; D(A))} \nonumber \\
\leq & \; C\|(v, q, u)\|^2_{W(\Phi_0)} \ .
\end{align}

This shows that $\mathcal{A}$ is well defined. On the other hand, that $\mathcal{A}$ is continuous is easy to prove using similar arguments; for brevity, we omit the details.
\end{proof}

\begin{lemma} {The mapping $\mathcal{A}  :  W(\Phi_0) \mapsto Z$ is continuously differentiable.}\label{L5}
\end{lemma}\label{L32}
\begin{proof} {We will present the proof when $N = 3$.} The proof for $N = 2$ is similar.

Let us first see that $\mathcal{A}$ is G-differentiable at any $(v, q, u) \in W(\Phi_0)$ and let us compute the G-derivative $\mathcal{A}'(v, q, u).$

Thus, let us fix $(v, q, u)$ and let us fix $(v', q', u') \in W(\Phi_0)$  and  $\sigma > 0.$ Let us denote by $\mathcal{A}_1$ and $\mathcal{A}_2$ the components of $\mathcal{A}$:
   $$
\mathcal{A}_1(v, q, u):= v_t + (v \cdot \nabla)v - (\nu + c_{\nu}\Phi_0) \Delta v + \nabla q - u \mathds{1}_{\omega}, \ \  \mathcal{A}_2(v, q, u):= v(\cdot , 0).
   $$
   We have:
   \begin{align}\label{3.43}
&\frac{1}{\sigma} \biggl[ \mathcal{A}_1((v, q, u) + \sigma(v', q', u')) - \mathcal{A}_1(v, q, u) \biggr] \nonumber  \\
& = v'_{t} \!+\! (v' \cdot \nabla)v \!+\! \sigma(v' \cdot \nabla)v' \!-\! (\nu + c_{\nu}\Phi_0) \Delta v' \!+\! \nabla q' \!-\! u' 1_{\omega} \!+\! (v \cdot \nabla)v'.
  \end{align}

   Let us introduce the linear mapping $D\mathcal{A} : W(\Phi_0) \mapsto Z$, with  $D\mathcal{A} = (D\mathcal{A}_1, D\mathcal{A}_2)$ and
   \begin{align}\label{3.36}
& D\mathcal{A}_{1}(v', q', u') = v'_{t} \!+\! (v \cdot \nabla)v' \!+\! (v' \cdot \nabla)v \!-\! (\nu \!+\! c_{\nu}\Phi_0) \Delta v' \!+\!  \nabla q' \!-\! u' 1_{\omega}\ ,
   \end{align}
and
   \begin{align} \label{3.37}
D\mathcal{A}_{2}(v', q', u') =  v'(\cdot, 0).
   \end{align}

From the definition of the spaces $W(\Phi_0)$, \ $Z$ and \eqref{3.36}--\eqref{3.37}, it becomes clear that $D\mathcal{A} \in \mathcal{L}(W(\Phi_0); Z).$ Furthermore,
\begin{align}
&\frac{1}{\sigma} \biggl[ \mathcal{A}((v, q, u) + \sigma(v', q', u')) - \mathcal{A}(v, q, u) \biggr] \rightarrow D\mathcal{A}(v, q, u)(v', q', u')\nonumber
\end{align}
strongly in $Z$ as $\sigma\rightarrow 0.$

Indeed, using the estimate of $M_2$ in (\ref{3.33a}), we get
\begin{align}
\frac{1}{\sigma} \bigl\| \mathcal{A}_{1}((v, q, u)&+ \sigma(v', q', u')) \!-\! \mathcal{A}_{1}(v, q, u) \!-\! D\mathcal{A}_{1}(v, q, u)(v', q', u')\bigr\|_{L^2(\tilde{\eta}^2; Q)} \nonumber \\
& = \|  \sigma(v'\cdot \nabla) v' \|_{L^2(\tilde{\eta}^2; Q)} \; \leq C \sigma \; \|(v', q', u')\|^2_{W(\Phi_0)} \rightarrow 0, \nonumber
\end{align}
while $\mathcal{A}_{2}$ is linear and continuous from $W(\Phi_0)$ into $V$.

Thus, $\mathcal{A}$\;is\;$G$-differentiable at any\; $(v, q, u) \in W(\Phi_0)$, with $G$-derivative
\begin{equation}\label{3.338a}
\mathcal{A}'(v, q, u) = D\mathcal{A} \ .
\end{equation}

Now, let us prove that the mapping $(v, q, u) \mapsto \mathcal{A}'(v, q, u)$ \;is continuous from $W(\Phi_0)$ into $\mathcal{L}(\Phi_0; Z) \ .$
As a consequence, in view of classical results, we will have that $\mathcal{A}$ is not only $G$-differentiable but also $F$-differentiable and $C^1$ and we will get the desired
result, see for instance Theorem 1-2, p.~21 in \cite{T.L.S}.

Thus, let assume that\; $(v^n, q^n, u^n) \rightarrow (v, q, u)$ \;in $W(\Phi_0)$ \;and let us check that
\begin{align}
& \| (D\mathcal{A}(v^n, q^n, u^n) - D\mathcal{A}(v, q, u))\;(v', q', u')\|^2_{Z} \leq \epsilon_{n} \| (v', q', u')\|^2_{W(\Phi_0)} ,\label{3.51}
\end{align}
for all $(v, q, u) \in W(\Phi_0)$  for some $\epsilon_n \rightarrow 0.$

Arguing as in (\ref{3.41}), the following holds:
\begin{align}\nonumber
 \| & (D\mathcal{A}_1(v^n, q^n, u^n) - D\mathcal{A}_1(v, q, u))\;(v', q', u')\|^2_{L^2(\tilde{\eta}^2; Q)} \nonumber\\
& \leq 2 \iint_Q \tilde{\eta}^2 |(v' \cdot \nabla) (v^n - v)|^2 \;dx\;dt + 2 \iint_Q \tilde{\eta}^2 |((v^n - v) \cdot \nabla)v'|^2 \;dx\;dt  \nonumber \\
& \leq 2 \int^T_0 \zeta \gamma^3 \|\nabla v'\|^2 \|\nabla (v^n -  v)\| \; \|\Delta (v^n - v) \|\;dt  \\
& + 2 \int^T_0 \zeta \gamma^3 \|\nabla(v^n - v)\|^2 \|\nabla v'\| \; \|\Delta v'\| \;dt \label{3.52} \nonumber \\
& \leq C \epsilon_{1,n} \;\|(v', q', u')\|^2_{W(\Phi_0)} \ , \nonumber
\end{align}
where
$$ \epsilon_{1,n} := \int^T_0 \zeta \gamma \|\nabla (v^n - v)\|\;\|\Delta (v^n - v)\| \;dt + \biggl( \sup_{[0,T]}\gamma(t) \|\nabla (v^n - v)\| \biggr)^2 \ .$$

Noting that $\epsilon_{1,n} \rightarrow 0 $  as $n \rightarrow \infty$ and recalling again that $\mathcal{A}_2$ is linear and continuous, we deduce that (\ref{3.51}) is satisfied.
\end{proof}

\begin{lemma}{Let $\mathcal{A}: W(\Phi_0) \mapsto Z$ be} the mapping defined by \textsc{(\ref{3.34})}.\label{L6} Then the linear mapping $\mathcal{A}' (0, 0, 0)$ is an epimorphism.
\end{lemma}\label{L33}

\begin{proof} Let us fix $(f, v_0) \in Z$. From Proposition \ref{P1}, we know that there exists $(v, q, u)$ satisfying (\ref{2.1a}), (\ref{2.23}) and (\ref{2.24}). From the usual regularity results for Stokes-like systems, we have $v \in L^2(0,T;D(A))$ and $q \in L^2(0,T;H^1(\Omega))$. Consequently $(v, q, u) \in W(\Phi_0)$ and, in view of (\ref{3.36}), (\ref{3.37}) and (\ref{3.338a}),
 $$\mathcal{A}' (0, 0, 0)(v', q', u')\; = \; (f,v_0) ,$$
 that is,
    \begin{equation} \nonumber
\left\{
\begin{array}
[c]{ll}%
\vspace{.2cm}
v_{t} - (\nu+c_{\nu}\Phi_{0})\Delta v + \nabla q = u1_{\omega} + f \; \; & \mbox{in}\;\;Q,\\
\vspace{.2cm}
\nabla\cdot{v} = 0 & \mbox{in}\;\;Q,\\
\vspace{.2cm}
v = 0 & \mbox{on}\;\;\Sigma,\\
\vspace{.2cm}
v(x,0) = v_{0}(x) \;\  & \mbox{on}\;\;\Omega.\\
\end{array}
\right. \\
   \end{equation}
This shows that $\mathcal{A}' (0, 0, 0)$ is surjective and ends to proof.
\end{proof}

In accordance with Lemmas~\ref{L4}, \ref{L5} and~\ref{L6}, we can apply Theorem~\ref{T3} and deduce that, at least when $(f,v_0)$ belongs to a neighborhood of the origin in $Z$ of the form $B_{\epsilon(M)}(0,0)$, the equation (\ref{3.39}) possesses a solution $(v,q,u)=S(f,v_0;\Phi_0),$ with
\begin{equation}\label{3.40a}
\|(v, q, u)\|_{W(\Phi_0)} \leq K(M)\|(f, v_0)\|_{Z} \ .
\end{equation}
 In particular, (\ref{3.31a}) is locally null-controllable.

Now, let $\phi_{00}$ and $k_0$ satisfy (\ref{1.4a}). Let us set
    \begin{equation}\label{3.40b}
\begin{array}
[c]{ll}%
M=&2\biggl(\displaystyle \phi_{00} + \frac{(a-2)T}{|\Omega|} \|k_0\|_{L^1(\Omega)} \biggr) \ , \ \; \; b_1= \biggl(\displaystyle \frac{|\Omega|}{2(a-2)c_{\nu}T} \biggr)^{1/2} \ , \; \; \\\\
\beta_0 =& \displaystyle\frac{\kappa \alpha^2 \phi_{00}}{\kappa \alpha^2 + \displaystyle\frac{2 c_0 \phi_{00}}{|\Omega|}\biggl( T\exp^T \|k_0\|^2 + c^2_{\nu}M^2(1+T\exp^T)\biggr)} \; .
\end{array}
   \end{equation}
\\
Let us assume that
\begin{equation}\label{3.40c}
\|v_0\|_{H^1_0} \leq \epsilon(M)\\
\end{equation}
and let us introduce the closed convex set
\begin{align}
G =  \Bigl\{ (\tilde{v} , \tilde{\phi})  &\in L^4 (0,T;V) \times C^0([0,T]) \ : \ \beta_0 \leq \tilde{\phi} \leq M , \ \nonumber\\
 & \ \iint_{Q}|D\tilde{v}|^2 \;dx\;dt \, \leq  \, b^2_1 \ , \ \iint_{Q}|D\tilde{v}|^4 \;dx\;dt \leq 1 \Bigr\} \nonumber
\end{align}
and the mapping $\mathcal{B} : G \mapsto L^4 (0,T;V) \times C^0([0,T])$, given as follows:
\begin{itemize}
  \item First, to each $(\tilde{v} , \tilde{\phi}) \in G$, we associate the solution to the semilinear parabolic system
  \begin{equation} \nonumber
\left\{
\begin{array}
[c]{ll}%
\displaystyle k_t + \tilde{v}\cdot \nabla k - (\kappa + c_0\tilde{\phi})\Delta k + \frac{k^2}{\tilde{\phi}} = c_{\nu} \tilde{\phi}|D\tilde{v}|^2 \; \; \ & \mbox{in}\;\;Q,\\
\vspace{.2cm}
\displaystyle \frac{\partial k}{\partial n} = 0   & \mbox{on}\;\;\Sigma,\\
k(x,0) = k_{0}(x) \;\  & \mbox{on}\;\;\Omega.\\
\end{array}
      \right.
   \end{equation}
  \item Secondly, we associate to $k$ the solution $\phi_0$ to the ODE problem (\ref{1.3}).
  \item Then, we associate to $\phi_0$ the quadruple $(v, q, u, \phi_0),$ where $(v, q, u)=S(0,v_0;\phi_0)$.
  \item Finally, we set $\mathcal{B}(\tilde{v},\tilde{\phi})=(v,\phi_0)$.
\end{itemize}

The mapping $\mathcal{B}$ is well defined, continuous and compact. Indeed, if we introduce the Hilbert space 
$$Y:=\{v \in L^2(0,T;D(A)) : v_t \in L^2(Q)^N\}$$
and we regard $\mathcal{B}$ as a mapping with values in~$Y \times C^1([0,T])$, it is obviously well defined and continuous. Since $Y\hookrightarrow L^4(0,T;V)$ with a compact embedding, the compactness of $\mathcal{B}$ is ensured.

On the other hand, if $\|v_0\|_{H^1_0}$ is sufficiently small, $\mathcal{B}$ maps $G$ into itself.\\
Indeed, from (\ref{1.3}) it is immediate that
\begin{align}
\phi_0(t) \ & \leq   \, \phi_{00}+\frac{(a-2)T}{|\Omega|}\|k\|_{L^{\infty}(0,T;L^1(\Omega))} \ \nonumber \\
& \leq \, \phi_{00}+\frac{(a-2)T}{|\Omega|}\bigl(\|k_0\|_{L^1}+c_{\nu}Mb_1^2\bigr)  \nonumber \\
& = \, M \nonumber
\end{align}
and

    \begin{equation}\nonumber
\begin{array}
[c]{ll}%
 \phi(t) \ &\geq \ \displaystyle\frac{\phi_{00}}{1+ \displaystyle\frac{2c_0 \phi_{00}}{|\Omega|} \iint_Q \displaystyle\frac{|\nabla k|^2}{|\alpha + k|^2}\;dx\;dt} \\\\
 \vspace{.1cm}
&\geq  \displaystyle\frac{\phi_{00}}{1+\displaystyle\frac{c_0 \phi_{00}}{|\Omega|\alpha^2 \kappa}\bigl( \|k_0\|^2 + \|k\|^2_{L^2(Q)} + c^2_{\nu}M^2 \iint_Q |D\tilde{v}| \;dx\;dt \bigr)} \\
&\geq \ \beta_0 \
\end{array}
   \end{equation}
for all $t \in [0,T].$
Moreover, since we have (\ref{3.40a}) with $f=0$, there exists $\tilde{\epsilon}$, only depending on $\Omega, \ \omega, \ T, \ \nu, \ c_{\nu}$ and $M$ such that, whenever $\|v_0\|_{H^1_0} \leq \tilde{\epsilon}$, one has
$$\iint_Q |Dv|^2 \;dx\;dt \ \leq \ b^2_1  \ \ \ \ \mbox{and} \ \ \ \ \iint_Q |Dv|^4 \;dx\;dt \ \leq \ 1 \ .  $$

As a consequence, we can apply \textit{Schauder's Theorem} to $\mathcal{B}$ and deduce the existence of a fixed point. This proves that \eqref{1.2}--\eqref{1.3} is partially locally null-controllable and ends the proof.

\section{Some additional comments and questions}
\noindent

The partially globally null controllability of \eqref{1.2}--\eqref{1.3} is an open question. It does not seem easy to solve and the answer is unknown even for the Navier-Stokes system with Dirichlet boundary conditions on $v$. Indeed, the smallness assumption on the data in \textit{Theorem}~\ref{T1} is clearly necessary if one tries to apply \textit{Theorem}~\ref{T3} or another result playing the same role. To prove a global result, we should have to make use of a global inverse mapping theorem, but this needs much more complicate estimates, that do not seem affordable.

Note that, with other or with no boundary conditions, global null controllability results have been established for Navier-Stokes and Boussinesq fluids by Coron \cite{Coron} and Coron and Fursikov \cite{Coron F} in the two-dimensional case, Fursikov and Imanuvilov \cite{F-III} and Coron, Marbach and Sueur \cite{4a} in the three-dimensional case. Accordingly, it is reasonable to expect results of the same kind when the PDEs in \eqref{1.2} are completed, for instance, with Navier-slip or periodic boundary conditions.

Another open question, in part connected to the previous one, concerns the partial exact controllability to the trajectories.

It is said that \eqref{1.2}--\eqref{1.3} is partially locally exactly controllable to trajectories at time $T$ if, for any solution $(\widehat{v}, \widehat{q})$ corresponding to a control $\widehat{u}$, there exists $\epsilon > 0$ such that, if
$$\|v_{0}-\widehat{v}(\cdot, 0)\|_{H^1_{0}} \leq \epsilon,$$
we can find controls $u \in L^2(\omega \times (0,T))^N$ and associated states $(v, q)$ satisfying
\begin{equation}\label{6.6}
v(x,T)=\widehat{v}(x,T) \ \ \; \ in \ \ \; \Omega.
\end{equation}

The previous property was established for the \textit{Navier-Stokes} and \textit{Boussinesq} fluids respectively in~\cite{15p} and~\cite{Enrique_Guerrero_Puel}. However, to our knowledge, it is unknown whether it holds for \eqref{1.2}--\eqref{1.3}. If one tries to apply arguments as those above, one finds at once a major difficulty: one is led to a system similar to \eqref{2.1b}, where linear nonlocal terms appear, for which observability estimates are not clear at all.

Let us also indicate that it would be interesting to see whether the arguments in \cite{Carreno-SGR} and \cite{E.M3} can be applied in this context to establish the partially local null controllability of \eqref{1.2}--\eqref{1.3} with $N-1$ scalar controls; when $N=3$, the controllability of \eqref{1.2}--\eqref{1.3} with one single scalar control is also a problem to consider, in view of the results in \cite{Lissy-Coron}.

The main result in this paper may be viewed as a first step in the controllability analysis of $k-\varepsilon$ models of turbulence. In this direction, let us indicate that it would be satisfactory to prove the local null controllability of~\eqref{1.1}. However, this seems difficult at present.

Another interesting aspect is the computation of ``good" null controls of~\eqref{1.2}--\eqref{1.3}.
If would be interesting to address an efficient iterative algorithm, able to produce a sequel of controls that converge to a null control for this system. This will be the goal of a forthcoming paper.

\

\noindent \textbf{Acknowledgements.} This paper is the result of a collaboration that took place during several visits of the first and third authors to the Institute of Mathematics of the University of Sevilla~(IMUS). The authors are indebted to IMUS for its assistance.




\end{document}